\documentclass[11pt,a4paper]{amsart}
\usepackage{amsmath,amssymb,latexsym}
\usepackage{amsthm}
\usepackage{bm}
\usepackage{overpic}
\usepackage{epstopdf}
\usepackage{color}
\theoremstyle{plain}% default
\newtheorem{theorem}{Theorem}[section]

\newtheorem{lemma}{Lemma}[section]
\theoremstyle{definition}

\theoremstyle{remark}
\newtheorem{remark}{Remark}[section]

\newcommand{\supp}{\mathop{\rm supp}}

\newcommand{\field}[1]{\mathbb{#1}}

\newcommand{\R}{\field{R}}

\newcommand{\N}{\field{N}}
\newcommand{\C}{\field{C}}

\renewcommand{\Im}{\mathop{\rm Im}}

\def\XXint#1#2#3{{\setbox0=\hbox{$#1{#2#3}{\int}$}
\vcenter{\hbox{$#2#3$}}\kern-.5\wd0}}

\title{On point-mass Riesz external fields on the real axis}

\date{\today}

\begin{document}

\centerline{\bf  \Large }
\author[D. Benko]{D. Benko}
\address{Department of Mathematics and Statistics, University of South Alabama, Mobile, AL 36688, USA}
\email{dbenko@southalabama.edu}

\author[P. Dragnev]{P. D. Dragnev $^{\dagger}$}
\address{Department of Mathematical Sciences, Purdue University, Fort Wayne, IN 46805, USA }
\email{dragnevp@pfw.edu}
\thanks{\noindent $^{\dagger}$ The research of this author was supported, in part, by a Simons Foundation grant no. 282207.}

\author[R. Orive]{R. Orive $^{*}$}
\address{Departmento de An\'{a}lisis Matem\'{a}tico, Universidad de La Laguna, 38200, The Canary Islands, Spain }
\email{rorive@ull.es}
\thanks{\noindent $^{*}$ The research of this author was supported, in part, by Ministerio de Ciencia e Innovaci\'{o}n under grant MTM2015-71352-P,
and was conducted while visiting PFW as a Scholar-in-Residence. The author wishes to thank the kind hospitality of P. Dragnev and the PFW}

\vskip 1 cm

\vspace{1cm} \maketitle

\begin{abstract}

The purpose of this work is twofold. First, we aim to extend for $0<s<1$ the results of one of the authors about equilibrium measures in the real axis in external fields created by point-mass charges for the case of logarithmic potentials ($s=0$).

Our second motivation comes from the work of the other two authors on Riesz $s$-equilibrium problems on finitely many intervals on the real line in the presence of external fields. They have shown that when the signed equilibrium measure has concave positive part on every interval, then the $s$-equilibrium support is also a union of finitely many intervals, with one of them at most included in each of the initial intervals. As the positive part of the signed equilibrium for point-mass external fields is not necessarily concave, the investigation of the corresponding $s$-equilibrium support is of comparative interest. Moreover, we provide simple examples of compactly supported equilibrium measures in external fields $Q$ not satisfying the usual requirements about the growth at infinity, that is, $\displaystyle \lim_{|x|\rightarrow \infty}\,Q(x) = \infty\,.$

Our main tools are signed equilibrium measures and iterated balayage algorithm in the context of Riesz $s$-equilibrium problems on the real line. As these techniques are not mainstream work in the field and can be applied in other contexts we highlight their use here.

\end{abstract}

\section{Introduction}
This paper is devoted to the study of Riesz $s$-equilibrium measures in the real axis in the presence of external fields created by fixed charges. We are dealing with Riesz $s$-potentials of the form $$U_s^{\sigma}(z) := \int \frac{d\sigma (x)}{|z-x|^s}\,,$$ for measures $\sigma$ supported in the real axis and $0<s<1$. The logarithmic potential $$U^{\sigma}(z) := U_0^{\sigma}(z) = - \int \log |z-x|\,d\mu (x)\,,$$
is the limit case of the Riesz potentials as $s\rightarrow 0^+$ (see e.g. \cite{Landkof}); hereafter, we usually refer to this case as log--case or, simply, as $s=0$.

%\textcolor{red}{At this point, it is convenient to recall the notion of \emph{admissible} external field in the case of Riesz $s$-potentials, since in this case the study of equilibrium problems in the presence of external fields is not so common as in the logarithmic one. Indeed, given a compact set $E \subset \R^d\,,\,d\geq 1\,,$ we say that a function $Q$ is an admissible external field in $E$, and we write $Q\in \mathcal{A}(E)\,,$ if it is a lower semi--continuous function in $E$, such that $\displaystyle \cp \left(x\in E\,:\,Q(x) < \infty \right)\,>\,0\,.$ Under these conditions, we have that (see \cite{BLW2} and \cite{BDS2014}, among others) there exists a unique equilibrium measure, which is uniquely characterized by the Frostman conditions and whose support, $S_Q = \supp \mu_Q$ is contained in the set $\displaystyle \{x\in E\,:\,Q(x) \leq M\} \,,$ for some $M < +\infty$.         }

More precisely, we are interested in equilibrium problems with external fields of the form
\begin{equation}\label{fixed}
Q(x) := Q(x;s,q,z)  = U^{-\,q\, \delta_{z}}(x)\,= -\,{q}{|x-z|^{-s}}\,,0<s<1\,,
\end{equation}
where, as usual, $\delta_a$ denotes the Dirac Delta measure at the point $a\in \C$, $q>0$, and $z \in \C \setminus \R$, that is, the external field is due to the action of an ``attractor'' (negative charge) located outside the real axis. In the sequel, we assume without loss of generality that $z = b i\,,\,b\neq 0\,.$

The focus of our study is the existence (and uniqueness) of the equilibrium measure $\mu_Q = \mu_{Q,s}$ that solves the \emph{Gauss variational problem} or the minimization of the weighted energy
\begin{equation}\label{minenergy}
I_Q (\nu) = I_{Q,s}(\nu):= I_s(\nu) + \int Q\,d\nu = \int \int \,\frac{1}{|x-y|^s}\,d\nu(x) d\nu(y)\,+\,2\int Q(x)\,d\nu(x)\,.
\end{equation}
It is well known that the minimizer, if it exists, is characterized by the variational inequalities, also known as Frostman's conditions,
\begin{equation}\label{equilibrium}
U_s^{\mu_Q}(x) + Q(x) \begin{cases} \geq F_Q\,,\;\text{q.e.}\;  \R \\ \leq F_Q\,,\;\text{everywhere on}\;  S_{Q} := \supp \mu_Q \, .\end{cases}
\end{equation}
Here ``q.e.'' means that the identity holds outside of a (possible) set of zero $s$-capacity. For the proof that \eqref{equilibrium} uniquely characterizes $\mu_Q$ in case it is compactly supported, see e.g. \cite[Theorem 1.2]{chafai} (there, the authors impose that $Q(x)\rightarrow \infty$ as $|x|\rightarrow \infty$, but it is easy to check that that assumption is not essential to prove this part). The compactness of the equilibrium support $S_Q$ is of particular interest. Since the external fields \eqref{fixed} are clearly continuous in the real axis, the first inequality in \eqref{equilibrium} holds everywhere on $\R$, and therefore equality takes place everywhere on $S_{Q}$. It is also known that since the Riesz kernel is strictly definite positive, the equilibrium measure $\mu_Q$, if it exists, is necessarily unique (see e.g. \cite{Zorii1}).

In the logarithmic case there is a broad array of applications to approximation theory (e.g. asymptotic behavior of orthogonal and Heine--Stieltjes polynomials, Pad\'{e} approximants), Random Matrix models, and to other fields related to physical applications (continuum limits of Toda lattice, Soliton theory,...). These applications have motivated a detailed study of equilibrium problems in the presence of a variety of external fields in the complex plane and, in particular, in the real axis (see e.g. the monograph \cite{Saff:97} and the recent paper \cite{MOR2015}, as well as the references therein). However, the situation is different in the case of Riesz $s$--potentials ($0<s<1$). In this last cases, mainly because the important applications to electrostatics and best-packing problems, the natural conductor to study the equilibrium problems is the sphere $\mathbb{S}^d$ in $\R^{d+1}$; see \cite{BDS2009}, \cite{BDS2014} and \cite{DrSa2007}, and the references therein, among many others (however, in a few papers like \cite{MMRS} and \cite{BD2012}, the case of the real axis has been also considered). Roughly speaking, most of the external fields considered have been potentials of single point masses or others closely related to them, like the axis-supported ones \cite{BDS2009}. Thus, in general, the role of the real axis for the logarithmic case has been played by the unit $d$-dimensional sphere, with the spherical caps being the counterpart of the real intervals. One of the aims of this paper is to extend the analysis carried out in the logarithmic case for classical external fields in the real axis like those due to the action of atomic measures to the case of general Riesz potentials.

The outline of the paper is as follows. Section 2 is devoted to announce and comment the main results of the paper, while in Section 3 some background and preliminary results about the balayage and the signed equilibrium measures, both being important tools in finding the (positive) equilibrium measure, are gathered. Finally, the proofs of the main results are given in Section 4.

\section{Main results}

Throughout this section the main results of the paper will be presented. We start with the existence of the equilibrium measure and the possible compactness of its support.
\begin{theorem}\label{thm:compint}
For the Gauss variational problem on the real axis in the external field \eqref{fixed}, we have the following.
\begin{itemize}
\item[(1)] For $q>1$, there exists a unique equilibrium measure and its support, $S_Q$, is a compact interval of the real axis.

\item[(2)] For $q<1$, the equilibrium measure does not exist.

\end{itemize}
\end{theorem}
\begin{remark}\label{rem:realproved}

It is well known that external fields that satisfy the usual growth condition $\displaystyle \lim_{|x|\rightarrow \infty}\,Q(x) = +\infty\,$ have compactly supported equilibrium measures (see, among others, \cite{abey}, \cite{chafai} and \cite{leble}). As far as we know, these are the first examples in the literature that assure compactness of the equilibrium support without the said growth condition. For the study of the solvability of Gauss variational problems within a more general context, see \cite{Zorii1}--\cite{Zorii2}, among other papers by N. V. Zorii.
\end{remark}

\begin{remark}\label{weakly}
The case $q=1$ will be considered in the next section (see Lemma 3.1). In that case, the equilibrium measure exists and agrees with the balayage of the point mass $\delta_z$ onto the real axis, that is, $\displaystyle \mu_Q = Bal_s(\delta_z,\R)$ (for the precise notation and definition, see Section 3 below). Of course, now the support is unbounded. This represents the borderline case between admissible cases (that is, when a compactly supported equilibrium measure exists, following \cite{Saff:97}) and the non-admissible ones. Utilizing similarity with the logarithmic potential theory, we shall refer to it as a ``weakly admissible'' case (see \cite{BLW}, \cite{Hardy}--\cite{Hardy-Kuijlaars}, \cite{OSW} and \cite{Simeonov} for the logarithmic case).
\end{remark}

Our second result deals with the expression of the density of the equilibrium measure, when it is compactly supported, and the precise determination of the endpoints $\pm \widetilde{a}$.
\begin{theorem}\label{thm:single}
If $q>1$, $S_Q = [-\widetilde{a},\widetilde{a}]$, and its density is given by
\begin{equation}\label{densityeqmeas}
\mu_Q'(x) = \,\frac{q\,b^{1-s}}{\mathcal{B}\left(\frac{1}{2},\frac{1-s}{2}\right)\,\mathcal{B}\left(\frac{1+s}{2},\frac{1-s}{2}\right)}\,
\left(\frac{\mathcal{B}\left(\frac{1+s}{2},\frac{1-s}{2}\right)}{(x^2+b^2)^{1-\frac{s}{2}}} - \,\frac{I_{\widetilde{a}}(\widetilde{a})- I_{\widetilde{a}}(x)}{(\widetilde{a}^2-x^2)^{\frac{1-s}{2}}}\right)\,,|x|<\widetilde{a}\,,
\end{equation}
where
\begin{equation}\label{integralI}
\begin{split}
I_a (x) = I_a (x,a,b,z,s):= & \int_{|x|>a}\,\frac{(a^2-t^2)^{\frac{1-s}{2}}}{(t^2+b^2)^{1-s/2}\,|x-t|}\,dt \\
= & \int_0^{\infty}\,\frac{u^{\frac{1-s}{2}}}{(u+a^2+b^2)^{1-\frac{s}{2}}}\,\frac{du}{u+a^2-x^2}\,.
\end{split}
\end{equation}
and
\begin{equation}\label{integralIa}
I_a (a) = \frac{\mathcal{B}\left(\frac{1}{2},\frac{1-s}{2}\right)}{\sqrt{a^2+b^2}}\,,
\end{equation}
with $\widetilde{a}=\widetilde{a}(s,q,b)\in (0,+\infty)$ being the (unique) solution of the equation $F'_s(a) = 0$, with $F_s(a)$ given by
    \begin{equation*}\label{FS(a)}
    \begin{split}
    & \mathcal{F}_s(a)\,  =\,\frac{\Gamma(1+s)}{2^s\,\Gamma\left(\frac{1+s}{2}\right)}\,a^{-s}\,\times \\ & \left(\Gamma\left(\frac{1-s}{2}\right)\,
 -\,q\,\frac{\sqrt{\pi}}{\Gamma(1+\frac{s}{2})}\left(\frac{a}{\sqrt{a^2+b ^2}}\right)^s\,_2F_1\left(\frac{s}{2},\frac{1+s}{2};1+\frac{s}{2};\frac{a^2}{a^2+b ^2} \right)\right)\,,
    \end{split}
    \end{equation*}
    where $_2 F_1(\alpha,\beta;\gamma;z))$ denotes, as usual, the hypergeometric function (see e.g. \cite[Ch. 15]{Abramowitz}).
\end{theorem}
%\begin{remark}\label{otherproofcompact}
%The proof of Theorem \ref{thm:single} (see Section 4.2 below) shows, in turn, that for $q>1$, $S_Q$ is a compact subset of the real axis. Nevertheless, we prefer to include a previous statement and proof of this fact (Theorem \ref{thm:compint}, part (1)) because of the techniques employed there, which are possibly applicable to more general settings (as pointed out in Remark \ref{rem:realproved}).
%\end{remark}
\begin{remark}\label{rem:comput}
It is possible to compute the critical value $\widetilde{a}$ given by Theorem \ref{thm:single}. Indeed, using the properties of the hypergeometric function (see \cite[Ch. 15]{Abramowitz}), and after some manipulations, we get that $\displaystyle \widetilde{a} = \,\sqrt{\frac{c}{1-c}}\,b \,,$ $c$ being the unique solution of the equation
\begin{equation*}\label{criticalc}
F_s(c)\,-\,(1-c)\,G_s(c) = \,\frac{\Gamma\left(\frac{1-s}{2}\right)\,\Gamma(1+\frac{s}{2})}{q\,\sqrt{\pi}}\,,
\end{equation*}
with $\displaystyle F_s(c) = \,_2F_1\left(\frac{s}{2},\frac{1+s}{2};1+\frac{s}{2};c\right)$ and $\displaystyle G_s(c) = \,_2F_1\left(1+\frac{s}{2},\frac{1+s}{2};1+\frac{s}{2};c\right)$.

Thus, preferably with the aid of some computer algebra software, we can obtain the critical value $\widetilde{a}$ of the endpoint. In this sense, it is easy to check that for $s=0.5$ and $q=5$, then $\displaystyle \frac{\widetilde{a}}{b}\, \approx 1.44227$; while for $s=0.5$ and $q=2$, we get $\displaystyle \frac{\widetilde{a}}{b}\, \approx 4.5233\,.$ The densities of both equilibrium measures are displayed in Fig. \ref{fig:sdensity}. Observe that the densities vanishes at the respective endpoints, as it happens in the logarithmic case (``soft edges''). This fact will be thoroughly analyzed below.
\end{remark}

\begin{figure}[h]
    \begin{center}
    \includegraphics[scale=0.4]{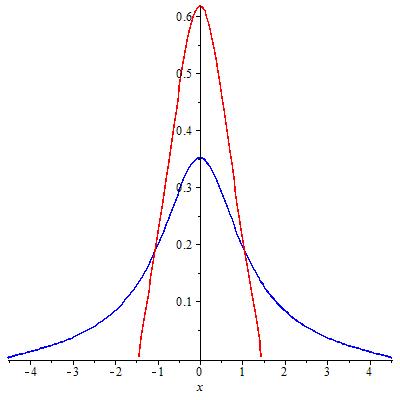}
    \end{center}
    \caption{Graph of the densities of the equilibrium measures $\mu'_Q$ for $s = 1/2, b = 1$, and $q = 5$ (red), $q=2$ (blue), on their respective supports. }
    \label{fig:sdensity}
\end{figure}

\begin{remark}\label{connectedness}
At first glance, the connectedness of the support of $\mu_Q$ might seem trivial, since the external field $Q$ has a single absolute minimum at the origin; but the fact that $Q$ is not convex on the whole real axis requires extra work. %Furthermore, in \cite[Th. 17]{BenkoTh} (see also \cite[Th. 2.4]{BDD2006}) the authors showed that the convexity of $\exp (Q)$ is a more general sufficient condition to assure the connectedness of $S_Q$; but this condition is not verified either.
\end{remark}
\begin{remark}\label{eqlog}
In the particular logarithmic case ($s=0$), the connectedness of the support is a consequence of the algebraic equation satisfied by the Cauchy Transform of the equilibrium measure (see \cite{OrSL2016} and previous \cite{MOR2015}--\cite{OrSL2015}). It can also be deduced from the fact that in this case the external field is weakly convex (see \cite[\S 3]{Benko2003}). Indeed, in this case, we have the following expression for the density of the equilibrium measure in the external field $Q$ (provided $q>1$),
\begin{equation*}\label{logdensity}
\mu'_Q (x) = \,\frac{\sqrt{\widetilde{a}^2-x^2}}{\pi\,(x^2+b^2)}\,,\;x\in (-\widetilde{a},\widetilde{a})\,.
\end{equation*}
where $$\widetilde{a} = \,\frac{\sqrt{2q-1}}{q-1}\,b\,.$$
\end{remark}
\begin{remark}\label{rem:alternative}
In Section 3 below, an alternative expression for the density of the equilibrium measure will be found, which will be used in the proof of Theorem 2.2 (see Section 4 below), namely,
\begin{equation}\label{densaltern}
\begin{split}
\mu_Q'(x)  & =\,q\,Bal'_s(\delta_z,[-\widetilde{a},\widetilde{a}])(x) - (qm_{\widetilde{a}}-1)\,\omega'_{s,[-\widetilde{a},\widetilde{a}]}(x) \\
& =\,q\,\frac{|\Im (z)|^{1-s}}{\mathcal{B}\left(\frac{1}{2},\frac{1-s}{2}\right)}\,\left(\frac{1}{(x^2+b^2)^{1-s/2}} + \,\frac{I_{\widetilde{a}}(x)}{\mathcal{B}\left(\frac{1-s}{2},\frac{1+s}{2}\right)\,(\widetilde{a}^2-x^2)^{\frac{1-s}{2}}}\right) \\
& -\,\frac{qm_{\widetilde{a}}-1}{\widetilde{a}^s\,\mathcal{B}\left(\frac{1}{2},\frac{1+s}{2}\right)\,
(\widetilde{a}^2-x^2)^{\frac{1-s}{2}}}\,,
\end{split}
\end{equation}
where the function $I_a (x)$ was given in \eqref{integralI}, $\omega_{s,[-\widetilde{a},\widetilde{a}]}$ is the equilibrium (Robin) measure of the interval $[-\widetilde{a},\widetilde{a}]$, $Bal_s (\delta_z,[-\widetilde{a},\widetilde{a}])$ denotes the balayage of the point mass $\delta_z$ onto $[-\widetilde{a},\widetilde{a}]$ and $m_{\widetilde{a}}\in (0,1)$ is its mass, i.e. $$m_{\widetilde{a}} = \|Bal_s (\delta_z,[-\widetilde{a},\widetilde{a}])\|\,$$ (it is known that, in general, the balayage of a measure onto a compact set entails a mass loss; see Section 3 below).
Now, comparing \eqref{densityeqmeas} and \eqref{densaltern} leads us to compute the precise mass loss for the balayage of the point mass $\delta_z$ onto the support of $\mu_Q$, the interval $[-\widetilde{a},\widetilde{a}]\,.$
Indeed, this comparison yields the following expression for the mass loss,
\begin{equation}\label{masslosssing}
1 - m_{\widetilde{a}} = 1 -\,\frac{1}{q}\,(1 + f(s)\,h(d,s))\,,\;0<s<1\,,\;q>1\,,\;b >0\,,\;d = \,\frac{\widetilde{a}}{b}\,> 0\,,
\end{equation}
\begin{equation*}\label{gh}
f(s) = \,\frac{\mathcal{B}\left(\frac{1}{2},\frac{1+s}{2}\right)}{\mathcal{B}\left(\frac{1-s}{2},\frac{1+s}{2}\right)}\,,\;h(d,s) = \,\frac{d^s}{\sqrt{1+d^2}}\,.
\end{equation*}
Observe that $$\lim_{s\rightarrow 0^+}\,f(s) = 1\,,\;\lim_{s\rightarrow 1^-}\,f(s) = 0$$
and that $h$ is an increasing function of $d$ for $\displaystyle |d| > \,\frac{s}{1-s}\,,$ while it is decreasing for $\displaystyle |d| < \,\frac{s}{1-s}\,.$

From \eqref{masslosssing}, we immediately obtain that for $s=0.5$ and $q=5$, the mass loss of the balayage $1-m_{\widetilde{a}} \approx 0.431$, while for $q=2$ we get $1-m_{\widetilde{a}} \approx 0.252$. These results seem natural: fixing $s$ and $b$ (the distance of the point mass to the real axis), a bigger mass $q$ provides a smaller support and, thus, a bigger mass loss.
\end{remark}
Our last main result is about the behavior of the density \eqref{densityeqmeas} at the endpoints of the support.
\begin{theorem}\label{thm:endbehavior}
\begin{equation}\label{endbehavior}
\eta_Q'(x) = O\left(|x\mp \widetilde{a}|^{\frac{1+s}{2}}\right)\,,\;x\rightarrow \widetilde{a}^{\pm}.
\end{equation}
\end{theorem}
\begin{remark}\label{generic}
The result in Theorem \ref{thm:endbehavior} seems to show that the generic behavior at the (soft) endpoints of the Riesz $s$--equilibrium measures in the presence of sufficiently regular external fields will be given by \eqref{endbehavior}. If it is indeed true in a general context, it would extend the well known result for the logarithmic case ($s=0$), about the generic ``square root'' behavior at the endpoints of the equilibrium density, as it was established in \cite{KuML 00}.
\end{remark}

\section{Balayage and signed equilibrium measures}

Throughout this section we deal with the balayage and signed equilibrium measures related to the problem we are addressing.

First, recall that given a closed set $F\subset \C$ and a measure $\sigma$, the measure $$\widehat{\sigma}:=\widehat{\sigma_s}= Bal_s \,(\sigma,F)$$ is said to be the Riesz $s$-balayage of $\sigma$ onto $F$ if $\supp \widehat{\sigma} \subseteq F$ and
\begin{equation}\label{balayage}
U_s^{\widehat{\sigma}}(z) = U_s^{\sigma}(z)\;\;\text{q.e. on}\;\;F\,,\;\;U_s^{\widehat{\sigma}}(z) \leq U_s^{\sigma}(z)\;\;\text{on}\;\;\C\,.
\end{equation}
In the logarithmic case ($s=0$), the balayage $\widehat{\sigma} = \widehat{\sigma}_0$ preserves the total mass: $\|\widehat{\sigma}\| = \|\sigma\|\,$, (and this is no longer true for $0<s<1$); but \eqref{balayage} is modified in the sense that
\begin{equation*}\label{logbalayage}
U^{\widehat{\sigma}}(z) = U^{\sigma}(z) + C\;\;\text{q.e. on}\;\;F\,,\;\;U^{\widehat{\sigma}}(z) \leq U^{\sigma}(z) + C\;\;\text{on}\;\;\C\,,
\end{equation*}
where $C=0$ if $\C \setminus F$ is a compact set.

Our first auxiliary result is the following one about the $s$-balayage of a point mass onto the real axis.
\begin{lemma}\label{lem:realaxis}
Let $z\in \C \setminus \R$ and denote by $\delta_z$ the Dirac Delta at the point $z$. Denote also $\widehat{\delta}_z:= Bal_s \,(\delta_z,\R)\,,\;0\leq s<1$. Then,
\begin{itemize}
\item The density of $\widehat{\delta}_{z}$ is given by
\begin{equation}\label{balreal}
\widehat{\delta}'_z (x) = \frac{|\Im (z)|^{1-s}}{B(\frac{1}{2},\frac{1-s}{2})\,|x-z|^{2-s}}\,,\;x\in \R.
\end{equation}
\item No mass loss takes place in this case, that is,
\begin{equation}\label{noloss}
\|\widehat{\delta}_z\| = \|\delta_z\| = 1\,,\;0\leq s<1.
\end{equation}
\end{itemize}
\end{lemma}
\begin{proof}

The proof is based on a stereographic projection and its associated Kelvin Transform (see \cite[Sec. IV.5]{Landkof}). Indeed, as in \cite[Lemma 3.4]{BDT2012}, let us take inversion with respect to the circle with center $z$ and radius $\sqrt{2 |\Im (z)|}$, in such a way that the image of the real axis is a circle $K$ with radius $1$ and $z$ as its north pole. Then, if the image of a generic point $x\in \R$ is denoted by $u\in K$, we have that
\begin{equation}\label{invrelat}
|x-z||u-z| = 2 |\Im (z)| \,,\; |du| = \frac{2 |\Im (z)| dx}{|x-z|^2}\,.
\end{equation}
In addition, if $y\in \R$ is another arbitrary point whose image is denoted by $v\in K$, the following relation between the distances holds,
\begin{equation}\label{distrelat}
|v-u| = \frac{2 |\Im (z)| |y-x|}{|x-z||y-z|}\,.
\end{equation}
On the other hand, it is well known that for $0\leq s<2$ the normalized arc-length measure, $d\omega(w) = \,\frac{|dw|}{2\pi}\,,$ is the $s$-equilibrium measure of a circle of unit radius, like $K$ (see e.g. \cite{Landkof}). This means that $\displaystyle V_s^{\omega}(w) = const\,=\,W_s\,,w\in K\,,$ where this constant is given by (see e.g. \cite{BDS2009})
\begin{equation}\label{circle}
W_s = W_s(S^1) = \,\frac{\mathcal{B}\left(\frac{1}{2},\frac{1-s}{2}\right)}{2^s \pi}\,.
\end{equation}

First, consider the case where $0<s<1\,.$ Now, in order to prove \eqref{balreal}, denote $$ d\sigma(x) = \frac{ |\Im (z)| ^{1-s}}{\mathcal{B}\left(\frac{1}{2},\frac{1-s}{2}\right)}\,\frac{dx}{|x-z|^{2-s}}\,,\;x\in \R\,.$$ Then, for $y\in \R$ and denoting as above by $v$ its image by inversion, and making use of \eqref{invrelat}-\eqref{circle}, we have
\begin{equation*}\label{proofbalreal}
\begin{split}
V_s^{\sigma}(y) & = \frac{ |\Im (z)| ^{1-s}}{\mathcal{B}\left(\frac{1}{2},\frac{1-s}{2}\right)}\,\int_{\R}\,\frac{dx}{|y-x|^s|x-z|^{2-s}} = \frac{2^{s-1}}{\mathcal{B}\left(\frac{1}{2},\frac{1-s}{2}\right)}\,\int_K\,\frac{|du|}{|y-z|^s|v-u|^{s}} \\
& = \frac{2^s \pi}{\mathcal{B}\left(\frac{1}{2},\frac{1-s}{2}\right)\,|y-z|^s}\,\int_K\,\frac{|du|/(2\pi)}{|v-u|^{s}} = \frac{2^s \pi\,W_s}{\mathcal{B}\left(\frac{1}{2},\frac{1-s}{2}\right)\,|y-z|^s} = \frac{1}{|y-z|^s} = V_s^{\delta_z}(y)\,,
\end{split}
\end{equation*}
which establishes \eqref{balreal}. In the log--case, where $s=0$, it suffices to take into account the invariance of harmonic measures under Kelvin Transforms (see \cite[Sec. IV.5]{Landkof}) and the fact that the image of $z$ by the above inversion is the point at infinity. Thus, we have that the equilibrium measure of $K$ (which is the balayage of $\delta_{\infty}$ onto $K$), that is, $d\omega(u) = \,\frac{|du|}{2\pi}\,,$ is the Kelvin Transform of the balayage $Bal_0(\delta_z; \R)$. Therefore, the second expression in \eqref{invrelat} yields
$$d\,Bal_0(\delta_z; \R)(x)\,=\,\frac{|du|}{2\pi}\,=\,\frac{2 |\Im (z)| dx}{2\pi\,|x-z|^2}\,=\,\frac{ |\Im (z)| dx}{\pi\,|x-z|^2}\,,$$ which proves \eqref{balreal} for $s=0$.

In a similar fashion, we have for $0<s<1$,
\begin{equation*}\label{proofnoloss}
\begin{split}
\|\widehat{\delta}_{z}\| & = \frac{ |\Im (z)| ^{1-s}}{\mathcal{B}\left(\frac{1}{2},\frac{1-s}{2}\right)}\,\int_{\R}\,\frac{dx}{|x-z|^{2-s}} = \frac{2^{s-1}}{\mathcal{B}\left(\frac{1}{2},\frac{1-s}{2}\right)}\,\int_K\,\frac{|du|}{|v-u|^s} \\
& = \frac{2^s \pi\,W_s}{\mathcal{B}\left(\frac{1}{2},\frac{1-s}{2}\right)} = 1 = \|\delta_z\|\,,
\end{split}
\end{equation*}
which renders the proof of \eqref{noloss} (for $s=0$ it is well known that the balayage preserves the mass).

%\textcolor{red}{SUBSTITUTE THE PROOF OF NO MASS LOSS BY A DIRECT ONE: USING BETA FUNCTION!!!}

\end{proof}

\begin{remark}\label{rem:directproof}
The fact that there is no mass loss for the balayage onto the whole real axis may be also proved computing the integral above in a straightforward way by using the Beta function.
\end{remark}

The next lemma is the well-known \emph{superposition principle} (see e.g. \cite[Sec. IV.5, (4.5.5)]{Landkof}).
\begin{lemma}\label{lem:superp}
Let $F$ a closed subset of the complex plane and $\nu$ a measure supported on $\C$. Then,
$$d Bal_s \,(\nu,F) = d(\nu|_F) + \left(\int_{\C \setminus F}\,\frac{d Bal_s \,(\delta_t,F)}{du}\,d\nu(t)\right)du\,.$$
\end{lemma}
As noted in the Introduction, we may assume without loss of generality that $z = b i\,,\,b\neq 0\,.$ Thus, we are in the presence of a symmetric (with respect to the imaginary axis) external field, and the support of the equilibrium measure (if it exists) inherits that symmetry. First, we need the expression of the balayage of a point mass placed in the real axis outside a certain interval onto that interval. This is given in \cite{BD2012}.
\begin{lemma}\label{lem:balintreal}
Let $a>0$ and $t\in \R \setminus [-a,a]$. Then, we have:
\begin{equation*}\label{balintreal}
d Bal_s \,(\delta_t,[-a,a]) = \gamma_s\,\left(\frac{a^2-t^2}{a^2-x^2}\right)^{\frac{1-s}{2}}\,\frac{dx}{|x-t|}\,,
\end{equation*}
where $\gamma_s = \mathcal{B}\left(\frac{1+s}{2},\frac{1-s}{2}\right)^{-1}\,.$
\end{lemma}

As an immediate consequence of Lemmas \ref{lem:realaxis}-\ref{lem:balintreal}, we have the following
\begin{lemma}\label{lem:balint}
Let $a>0$ and $z=b i\,,\,b\neq 0$. Then,
\begin{equation}\label{balint}
d Bal_s \,(\delta_z,[-a,a]) = \,\frac{b^{1-s}}{\mathcal{B}\left(\frac{1}{2},\frac{1-s}{2}\right)}\,\left(\frac{1}{(x^2+b^2)^{1-s/2}} + \gamma_s\,\frac{I_a (x)}{(a^2-x^2)^{\frac{1-s}{2}}}\right)\,dx\,=\,f(x)dx\,,
\end{equation}
where $I_a (x)$ is given in \eqref{integralI}.
\end{lemma}

\begin{remark}\label{rem:endpoint}
Using \eqref{integralI} it is easy to study the behavior of the density of \eqref{balint} at the endpoints of the interval. Indeed, we have,
\begin{equation}\label{singend}
\begin{split}
\lim_{x\rightarrow \pm a^{\mp}}\,(a^2-x^2)^{\frac{1-s}{2}}\,f(x) & = \,\frac{\gamma_s\,b^{1-s}}{\mathcal{B}\left(\frac{1}{2},\frac{1-s}{2}\right)\,\sqrt{a^2+b^2}}\,\int_0^{\infty}\,\frac{v^{-\frac{1+s}{2}}\,dv}{(v+1)^{1-\frac{s}{2}}} \\
& = \frac{\gamma_s\,b^{1-s}}{\sqrt{a^2+b^2}}\,> \,0\,,
\end{split}
\end{equation}
and hence, the behavior of the density of the balayage measure at the endpoints is of the form $\displaystyle O\left(|x\pm a|^{-\frac{1-s}{2}}\right)\,.$
\end{remark}

Similar to the logarithmic case ($s=0$), for $0<s<1$ if the external field is a convex function on a certain interval $I\subset \R$, then the intersection of $S_Q$ (the support of the equilibrium measure in $Q$) and $I$ is a single interval; but not much more is known, in general, for the support of such Riesz $s$-equilibrium measures.

To study the equilibrium measure of a real interval in the presence of \eqref{fixed}, we will make use of the signed equilibrium measure in such external field.
Namely, for a closed subset $F$ of $\R$ (or $\C$) the $s$-signed equilibrium measure of $F$ associated with the external field $Q$ is a signed measure $\eta_{Q,F}$ supported on $F$, such that $\eta_{Q,F} (F) = 1$ and
\begin{equation}\label{defsigned}
V_s^{\eta_{Q,F}}(x) + Q(x) = C_{Q,F}\,,\;\text{q.e. on}\;F\,.
\end{equation}
If this signed equilibrium measure exists, then it is unique (see \cite[Lemma 2.1]{BDS2009}). The importance of it for our analysis relies on the fact that
\begin{equation}\label{signedsupp}
\supp \mu_{Q,F} \subseteq \supp \eta_{Q,F}^+\,,
\end{equation}
where $\mu_{Q,F}$ is the (positive) $s$-equilibrium measure of $F$ in $Q$ and $\eta_{Q,F}^+$ denotes the positive part of the signed measure in the usual decomposition $\eta_{Q,F} = \eta_{Q,F}^+ - \eta_{Q,F}^-\,$ (see \cite{KD}, where this result is established for $s=0$; the result remains valid for a general $s$, see \cite[Theorem 9]{BDS2014}). Therefore, we know that $\supp \mu_{Q,F}$ cannot intersect intervals where the signed measure $\eta_{Q,F}$ is not positive.

Another useful tool for our analysis is the so--called \emph{Mhaskar--Saff $\mathcal{F}_s$-functional}, which generalizes the Mhaskar--Saff functional (for the logarithmic potential see \cite[Sect. IV, eq. (1.1)]{Saff:97}). Namely, if the support $S_Q$ of the Riesz $s$-equilibrium measure of a conductor $\Sigma \subset \C$ in the external field $Q$ is a compact subset of $\Sigma$, then $S_Q$ minimizes, among all the compact subsets of $\Sigma$, the functional (see e.g. \cite{BDS2014})
\begin{equation}\label{MSfunct}
\mathcal{F}_s(K)\,:=\,W_s(K)\,+\,\int Q\,d \omega_{K}\,,
\end{equation}
where $\omega_K = \omega_{K,s}$ is the (unweighted) $s$-equilibrium measure of $K$ and $W_s (K)$ denotes its (unweighted) energy: that is, such that $U_s^{\omega_K} \equiv W_s(K)\,,$ everywhere in $K$. Observe that this Riesz version of the usual Mhaskar--Saff functional has opposite sign to the classical one. As for the connection with the signed equilibrium measure, take notice that if the signed equilibrium measure on a compact set $K$, $\eta_{Q,K}$, exists, then
\begin{equation*}\label{relsignedMSF}
\mathcal{F}_s(K)\,= C_{Q,K}\,,
\end{equation*}
with $C_{Q,K}$ being the equilibrium constant given in \eqref{defsigned}.

Now, it is easy to check that for the case we are dealing with in the current paper, the expression of the signed equilibrium measure for a real interval $[-a,a]$ in the presence of the external field \eqref{fixed} will be given by
\begin{equation}\label{signedfixed}
\eta_a = \eta_{Q,[-a,a]} = q\, Bal_s (\delta_{z},[-a,a])\,-\,(q\,m_a\,- 1)\,\omega_{[-a,a]}\,,
\end{equation}
where the expression of $Bal_s (\delta_{z},[-a,a])$ was given in \eqref{balint} and $$0< m_a = \|Bal_s (\delta_z,[-a,a])\| \leq 1$$ (recall again that the Riesz $s$-balayage of measures onto compact sets implies in general a mass loss (see \cite{BD2012})), and $\omega_{[-a,a]} := \omega_{s,[-a,a]}$ stands for the $s$-equilibrium measure of the interval $[-a,a]$, which is given by (see \cite{Landkof}, \cite{BD2012} and \cite{MMRS}),
\begin{equation}\label{sequilmeasint}
d\omega (x) = C_a\,\frac{dx}{(a^2-x^2)^{\frac{1-s}{2}}}\,,
\end{equation}
where
\begin{equation}\label{consteq}
C_a = C_{s,a} = \,\left((2a)^s\,\mathcal{B}\left(\frac{1+s}{2},\frac{1+s}{2}\right)\right)^{-1}\,=\,\left(a^s\,\mathcal{B}\left(\frac{1}{2},\frac{1+s}{2}\right)\right)^{-1}.
\end{equation}

Take into account, also, that $m_a$ is an increasing function of $a$, such that $\displaystyle \lim_{a\rightarrow 0^+}\,m_a = 0\,$ and $\displaystyle \lim_{a\rightarrow +\infty}\,m_a = 1\,.$ Now, from \eqref{signedfixed} and \eqref{singend}, we have for the behavior of the density of $\eta_a$ at the endpoints $\pm a$,
\begin{equation}\label{endbehsigned}
\lim_{x\rightarrow \pm a^{\mp}}\,(a^2-x^2)^{\frac{1-s}{2}}\,\eta'_a(x) = \,\frac{b^{1-s}}{\mathcal{B}\left(\frac{1-s}{2},\frac{1+s}{2}\right)\,\sqrt{a^2+b^2}}\,-\,
\frac{q\,m_a - 1}{a^s\,\mathcal{B}\left(\frac{1-s}{2},\frac{1+s}{2}\right)}\,.
\end{equation}
Suppose now that $q>1$. Then, as $a$ grows from $0$ to $+\infty$, the density of the signed equilibrium measure $\eta_a$ given in \eqref{signedfixed} evolves as shown in Figures \ref{fig:eta1}--\ref{fig:eta4}. Namely, for $a$ below a certain critical value $a_1 > 0\,,$ $\eta_a$ is, in fact, the sum of two positive measures (Fig. \ref{fig:eta1}), while for $a_1 < a < a_2$, for certain $a_2$, $\eta_a$ is the difference of two positive measures but is still a positive measure everywhere in $[-a,a]$ (Fig. \ref{fig:eta2}); thus, by uniqueness, for $0<a<a_2$, $\eta_a$ agrees with $\mu_a$, the (positive) equilibrium measure of $[-a,a]$ in $Q$. Conversely, for $a>a_2$, $\eta_a$ is no longer a positive measure everywhere in $[-a,a]$; indeed, by \eqref{endbehsigned} we have that the support of its negative part consists of two symmetric subintervals $[a',a]$ and $[-a,-a']\,,\,0<a'<a\,$ (Fig. \ref{fig:eta4}.)

Otherwise, if $q<1$, \eqref{signedfixed} is always a positive measure everywhere on $[-a,a]$, and thus, $\mu_{Q,[-a,a]}\equiv \eta_a\,,$ for any $a>0\,.$ Therefore, for $q<1$ the density of the signed (indeed, positive) equilibrium measure on $[-a,a]$ always diverges to $+\infty$ at the endpoints; see Fig. \ref{fig:qless1}, where the graphs of density of $\eta_a$ for $q = .5$ and $a=10$ and $a=20$ are displayed.

\begin{figure}
\centering
\begin{minipage}{.5\textwidth}
  \centering
  \includegraphics[width=.5\linewidth]{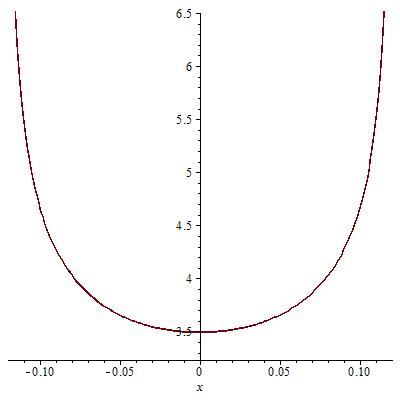}
  \caption{Density of $\eta_a$ for $q=5$ and $a = .12$.}
  \label{fig:eta1}
\end{minipage}%
\begin{minipage}{.5\textwidth}
  \centering
  \includegraphics[width=.5\linewidth]{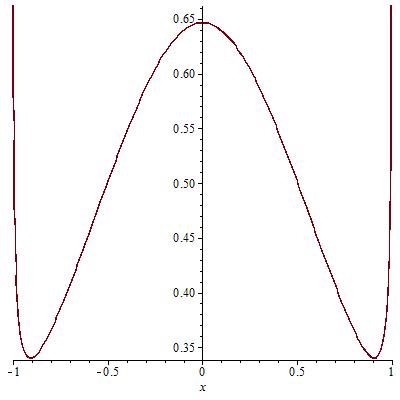}
  \caption{Density of $\eta_a$ for $q=5$ and $a = 1$}
  \label{fig:eta2}
\end{minipage}%
\end{figure}

\begin{figure}[h]
    \begin{center}
    \includegraphics[scale=0.3]{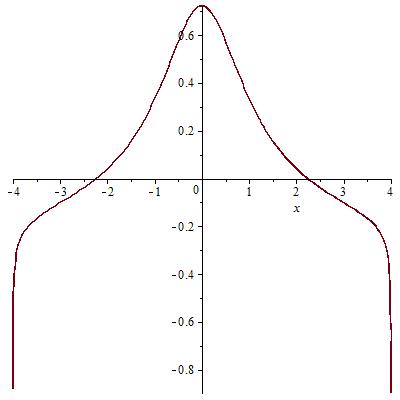}
    \end{center}
    \caption{Density of $\eta_a$ for $q=5$ and $a = 4 > a_2$. }
    \label{fig:eta4}
\end{figure}

\begin{figure}[h]
    \begin{center}
    \includegraphics[scale=0.4]{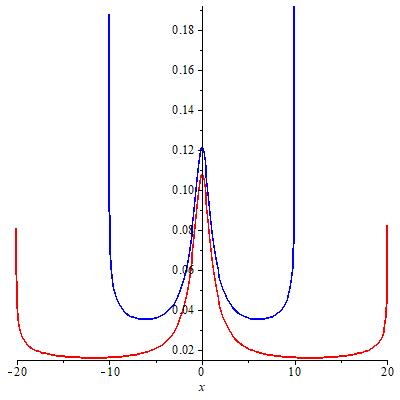}
    \end{center}
    \caption{Density of $\eta_a$ for $q=.5$ and $a = 10$ (blue) and $a = 20$ (red). }
    \label{fig:qless1}
\end{figure}

There is a very useful tool combining the two main concepts in this section, namely the signed equilibrium measure and the balayage of a measure, which may lead to determine the (positive) equilibrium measure. This is a recursive procedure known as the \emph{Iterated Balayage Algorithm} (hereafter, IBA). This method has been successfully applied in a number of problems in the case of logarithmic potentials (see \cite{KD}, \cite{DDK}, \cite{DK}, among others), and also, implicitly at least, for the case of Riesz $s$-potentials in \cite{BDS2009} (there, a continuous version of this algorithm was used).
Next, we include a short description of this method (for details in the log--case, see \cite{KD} and \cite{DragnevPhD}).

Namely, suppose that we are given a Gauss variational problem in the real axis (for Riesz $s$-potentials, $0<s<1$) in the presence of a certain external field $Q$ and it is known that the support $S_Q$ of the equilibrium measure $\mu_Q$ is compact. Suppose $\Sigma = \Sigma_0 $ is a compact subset of the real axis large enough so that $S_Q \subset \Sigma_0$, and let $\nu_0 = \eta_{Q,\Sigma_0}$ be the corresponding signed equilibrium measure, according to the definition \eqref{defsigned}. Then, if $\nu_0 = \nu_0^+ - \nu_0^-$ denotes its Jordan decomposition, then \eqref{signedsupp} yields that $\mu_Q \leq \nu_0^+$ and $$S_Q \subset \supp \nu_0^+ = \Sigma_1\,.$$ Now, we can compute the signed equilibrium measure $\nu_1$ in the external field $Q$ for the compact interval $\Sigma_1$ and we also have that $\mu_Q \leq \nu_1^+$ and $$S_Q \subset \supp \nu_1^+ = \Sigma_2 \subset \Sigma_1\,.$$ Iterating this procedure, a sequence of measures $\{\nu_k\}$ and other sequence of nested compact sets $\{\Sigma_k\}$ are generated so that $$S_Q \subset \Sigma_{k+1} = \supp \nu_k^+\,,\,k=0,1,\ldots$$
Obviously, the algorithm stops if at any time the signed measure $\nu_k$ is actually a positive one; in that case, we would have that $\mu_Q = \nu_k$.
Thus, it seems reasonable to expect that the sequence $\{\nu_k^+\}$ converges to the equilibrium measure $\mu_Q$, and it is indeed the case if certain conditions are satisfied; the hardest part is often to prove that the sequence of negative parts $\{\nu_k^-\}$ tends to zero as $k\rightarrow \infty$. This requires to control the limit set
\begin{equation}\label{sigmastar}
\Sigma^* = \cap_{k=1}^{\infty}\,\Sigma_k\,.
\end{equation}
In the proof of Theorem \ref{thm:single}, in the next section, this method will be applied.
\begin{remark}\label{rem:nameIBA}
In the log--case, each of the successive signed measures $\nu_{k+1}$ agrees with the balayage of the previous one, $\nu_k$, onto the support of its positive part $\Sigma_{k+1} = \supp \nu_k^+\,,\,k=0,1,\ldots$. Due to the mass loss in Riesz $s$--balayage we modify, the iterative scheme above by adding appropriate scaling of the regular $s$--equilibrium measure at each step.
%Reading the above description, the use of the word ``balayage'' in the name of the IBA method might be, at first sight, surprising, since it is not indeed used. This is due to the fact that in the log--case, each of the successive signed measures $\nu_{k+1}$ agrees with the balayage of the previous one, $\nu_k$ onto the support of its positive part $\Sigma_{k+1} = \supp \nu_k^+\,,\,k=0,1,\ldots$. As it was pointed out, in the general Riesz $s$--setting the balayage implies a mass loss, and that relationship is no longer valid. Anyway, we have preferred to preserve the original name for the adapted algorithm.
\end{remark}

To end this section, along with the family of signed equilibrium measures $\{\eta_a\}_{a>0}\,,$ we will also consider the one--parameter family of measures $\{\sigma_a\}_{a>0}\,,$ whose densities are given by
\begin{equation}\label{sigmaa}
\sigma'_a (x) = \,\frac{q\,b^{1-s}}{\mathcal{B}\left(\frac{1}{2},\frac{1-s}{2}\right)\,\mathcal{B}\left(\frac{1+s}{2},\frac{1-s}{2}\right)}\,
\left(\frac{\mathcal{B}\left(\frac{1+s}{2},\frac{1-s}{2}\right)}{(x^2+b^2)^{1-\frac{s}{2}}} - \,\frac{I_a (a)- I_a (x)}{(a^2-x^2)^{\frac{1-s}{2}}}\right)\,,|x|<a.
\end{equation}
with $I_a (x)$ given by \eqref{integralI}. First, we have the following
\begin{lemma}\label{lem:sigma1}
For any $a>0$, it holds:

\begin{itemize}

\item[(i)] $\sigma_a$ is a positive measure for any $a>0$

\item[(ii)] $U^{\sigma_a} (x) + Q(x) = E_a = -\,\frac{q\,b^{1-s}}{\sqrt{a^2+b^2}} = const < 0\,,\;|x|\leq a\,.$

\item[(iii)] $\|\sigma_a\|\in (0,q)\,,$ as $a$ varies through $(0,+\infty)$

\item[(iv)] $\displaystyle \lim_{|x|\rightarrow a^-}\,\sigma'_a (x) = 0\,.$

\end{itemize}

\end{lemma}

\begin{proof}

First, let us see that $\sigma$ is a positive measure in $(-a,a)$. To do it, take into account that \eqref{integralI} implies that
\begin{equation*}
\begin{split}
\frac{I_a (a)- I_a (x)}{(a^2-x^2)^{\frac{1-s}{2}}}\, & =\,(a^2-x^2)^{\frac{1+s}{2}}\,\int_0^{\infty}\,\frac{u^{-\,\frac{1+s}{2}}}{(u+a^2+b^2)^{1-\frac{s}{2}}}\,
\frac{du}{u+a^2-x^2} \\
\leq & \frac{1}{(a^2+b^2)^{1-\frac{s}{2}}}\,\int_0^{\infty}\,\frac{v^{-\,\frac{1+s}{2}}}{1+v}\,dv\,=\,\frac{B(\frac{1+s}{2},\frac{1-s}{2})}{(a^2+b^2)^{1-\frac{s}{2}}} \\
\leq & \,\frac{B(\frac{1+s}{2},\frac{1-s}{2})}{(x^2+b^2)^{1-\frac{s}{2}}}\,.
\end{split}
\end{equation*}
Second, it is necessary to check that the total (also called ``chemical'') potential $$U^{\sigma}(x)\,+\,Q(x) = U^{\sigma}(x)\,-\,\frac{q}{|x-z|^s} = E \equiv const$$ on $[-a,a]$. But it is immediate because \eqref{sigmaa}, \eqref{balint}, and \eqref{integralI} imply that
\begin{equation}\label{sigmaa1}
\sigma_a = q\,\left(Bal_s\,(\delta_z, [-a,a])\,-\,\frac{b^{1-s}\,I_a (a)}{\mathcal{B}\left(\frac{1}{2},\frac{1-s}{2}\right)\,\mathcal{B}\left(\frac{1+s}{2},\frac{1-s}{2}\right)\,C_s}\,\omega_{s,[-a,a]}\right)\,,
\end{equation}
where $\omega_{s,[-a,a]}$ denotes the (Robin) $s$-equilibrium measure of $[-a,a]$ (see \eqref{sequilmeasint}) and $C_s = C_{s,[-a,a]}$ is given by \eqref{consteq}. Indeed, taking into account the expression of the Riesz $s$-energy of the interval (i.e. its unweighted equilibrium energy)
\begin{equation}\label{energyint}
W_s(a) = \,\frac{\Gamma(\frac{1-s}{2})\,\Gamma(1+s)}{2^s\,\Gamma(\frac{1+s}{2})}\,a^{-s}
\end{equation}
and \eqref{integralIa}, we obtain that
\begin{equation*}\label{Fconst}
E = E_{s,q,b,a} = -\,\frac{q\,b^{1-s}}{\sqrt{a^2+b^2}}\,.
\end{equation*}
%Now, the inequality $$U^{\sigma}(x)\,+\,Q(x) = U^{\sigma}(x)\,-\,\frac{q}{|x-z|^s} \geq F\,,\,x\in \R$$ must be established.
As for part (iii), that $\|\sigma_a\| \in (0,q)$ immediately arises from \eqref{sigmaa1}, in such a way that $\displaystyle \lim_{a\rightarrow 0^+} \|\sigma_a\| = 0$ and $\displaystyle \lim_{a\rightarrow \infty} \|\sigma_a\| = q\,.$

To prove (iv), note that
\begin{equation}\label{limit}
\begin{split}
\frac{I_a (a)- I_a (x)}{(a^2-x^2)^{\frac{1-s}{2}}} & =\,(a^2 - x^2)^{\frac{1+s}{2}}\,\int_0^{\infty}\,\frac{u^{-\,\frac{1+s}{2}}}{(u+a^2+b^2)^{1-\frac{s}{2}}}\,
\frac{du}{u+a^2-x^2} \\
& = \,\int_0^{\infty}\,\frac{v^{-\frac{1+s}{2}}}{(v+1)\,((a^2-x^2)v + a^2+b^2)^{1-\frac{s}{2}}}\,dv
\end{split}
\end{equation}
and, taking into account that $$ \mathcal{B}\left(\frac{1+s}{2},\frac{1-s}{2}\right) = \int_0^{\infty}\,\frac{u^{-\frac{1+s}{2}}}{u+1}\,du\,,$$ \eqref{limit} easily yields  $$\lim_{|x|\rightarrow a^-}\,\sigma'_a (x)\,=\,0\,.$$

\end{proof}

In Fig \ref{fig:fig6} - Fig \ref{fig:fig7}, the plots of the densities of $\sigma_a$ for $q = 5$ and $b = 1$, corresponding to the cases $a = 5$ (in this case, $\|\sigma_a\| = 2.616107631$) and $a = .5$ ($\|\sigma_a\| = 0.1579953845$), are displayed.

\begin{figure}
\centering
\begin{minipage}{.5\textwidth}
  \centering
    \includegraphics[scale=0.3]{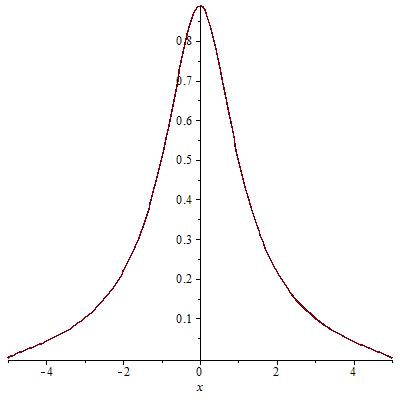}
    %\end{center}
    \caption{Density of $\sigma_a$ for $a = 5$. }
    \label{fig:fig6}
\end{minipage}%
\begin{minipage}{.5\textwidth}
  \centering
    \includegraphics[scale=0.3]{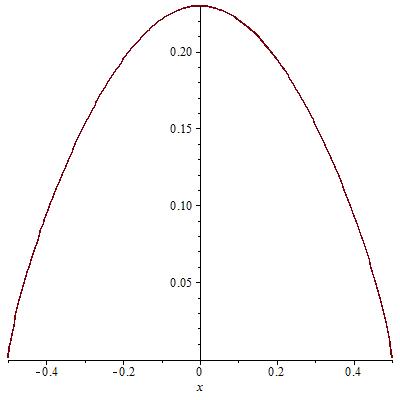}
    %\end{center}
    \caption{Density of $\sigma_a$ for $a = 0.5$. }
    \label{fig:fig7}
\end{minipage}%
\end{figure}

Finally, observe also that, comparing \eqref{sigmaa} with \eqref{signedfixed}, we have that:
\begin{equation*}\label{relationship}
\eta_a = \sigma_a + K\,\omega_{[-a,a]}\,,
\end{equation*}
for a certain constant $K = K(q,b,s,a)$. Indeed, this close relationship between the positive measure $\sigma_a$ and the signed measure $\eta_a$ will be exploited in the proof of Theorem \ref{thm:single}.

\section{Proofs of the main results}

\subsection{Proof of Theorem \ref{thm:compint}}

\textbf{(1)} As seen in Section 3, for $a$ large enough the signed equilibrium measure $\eta_{a} = \eta_{Q,[-a,a]}$ has a negative part supported on two symmetric pieces close to the endpoints of the interval, say $[-a,-b] \cup [b,a]$, with $0<b<a$ (see Fig. 4 and \eqref{endbehsigned}). Then, it is possible to apply the IBA (also described in Section 3) to find the equilibrium measure $\mu_{a} = \mu_{Q,[-a,a]}\,,$ starting from $\Sigma_0 = [-a_0,a_0] = [-a,a]$. In this way, as shown in the previous section, we can build a sequence of measures $\{\nu_k^+\}$ such that $S_Q  \subset \ldots \subset \Sigma_{k+1} = \supp \nu_k^+ \subset \ldots  \subset \Sigma_1 = \supp \nu_0^+\,,$ where $$\nu_k = \eta_{a_k} = \eta_{Q,[-a_k,a_k]}\,,$$ with $\Sigma_{k} = [-a_k,a_k] = \supp \nu_{k-1}^+\,,$ and $a_0 \geq a_1 \geq \ldots $

Now, we are concerned with finding the limit of the sequence $\{\nu_k\}$. But \eqref{endbehsigned} shows that the limit of the recursive algorithm above occurs when the critical value of $a$, for which
\begin{equation}\label{criticalma}
\frac{b^{1-s}}{\mathcal{B}\left(\frac{1+s}{2},\frac{1-s}{2}\right)\,\sqrt{a^2+b^2}}\,=\,\frac{q\,m_a - 1}{\mathcal{B}\left(\frac{1+s}{2},\frac{1}{2}\right)\,a^s}\,,
\end{equation}
is attained, for which the density of $\eta_a'$ vanishes at the endpoints, and no negative part of the signed equilibrium measure is supported in two symmetric pieces beside the endpoints. %Now, and before asserting that this critical case is the limit of the IBA process or, in other words, that we have found the limit set $\Sigma^*$ in \eqref{sigmastar}, and, consequently, that we have found the (positive) equilibrium measure, we need to be sure that the corresponding $\eta_a'$ is positive everywhere in $[-a,a]$.
But comparing with the expression of the density of $\sigma_a$ \eqref{sigmaa}, it is easy to check that the critical value of $a$ for which identity \eqref{criticalma} is attained, agrees with the one for which $\displaystyle \|\sigma_a\| = 1\,$ (see \eqref{sigmaa1}). Therefore, if we call $\widetilde{a}$ that critical value of $a$, we have that the IBA process converges to a unit measure, being positive everywhere on $[-\widetilde{a},\widetilde{a}]$ and, hence, $$\mu_{Q,a} \equiv \eta_{\widetilde{a}} \equiv \sigma_{\widetilde{a}}\,.$$
Thus, by the Frostman conditions we have that
\begin{equation}\label{FrostIBA}
V^{\sigma_{\widetilde{a}}}(x) + Q(x) \begin{cases} = C\,,\;& |x|\leq \widetilde{a}\,,\\ \geq C\,,\;& |x|\leq a\,, \end{cases}
\end{equation}
for some constant $C$. But since the limit measure of IBA is the same for any $a\geq \widetilde{a}$, \eqref{FrostIBA} also holds for any $a\in (\widetilde{a},+\infty)\,,$ and therefore, the Frostman inequality holds for any $x\in \R$ and, in turn, $\mu_Q \equiv \sigma_{\widetilde{a}}$. This renders the proof.

%\textcolor{red}{FOR PART (2) WE NEED FROSTMAN + LIMITS!!!}

\textbf{(2)} For the proof of the second part we need the following Lemma.
\begin{lemma}\label{lem:Fbehavior}
For the Gauss variational problem \eqref{minenergy} in the external field \eqref{fixed}, we have:

\begin{itemize}

\item[(a)] If there exists an equilibrium measure $\mu_Q$, satisfying Frostman conditions \eqref{equilibrium}, then it holds $F_Q \leq 0$.

\item[(b)] If $S_Q$ is unbounded, then necessarily $q=1$ (and, hence, $S_Q = \R$).

\end{itemize}

\end{lemma}

\begin{proof}
The proof of (a) is immediate in case $S_Q$ is a compact subset of the real axis. Indeed, it is enough to take limits as $|x|\rightarrow \infty$ in the inequality of \eqref{equilibrium}, and take into account that
$$\lim_{|x|\rightarrow \infty}\,V^{\mu_Q}(x) = \lim_{|x|\rightarrow \infty}\,\int\,\frac{d\mu_Q(y)}{|x-y|^s}\, = 0 = \,\lim_{|x|\rightarrow \infty}\,Q(x)$$ to conclude that $F_Q \leq 0$.

In case $S_Q$ is unbounded, it is enough to follow a \emph{monotone convergence} argument, using the sequence of truncated measures $\displaystyle \mu_n = \mu_Q |_{[-n,n]}\,,\,n\in \N\,,$ which converges to $\mu_Q\,,$ as $n\rightarrow \infty\,.$ It is clear that $\mu_n \leq \mu_{n+1}\,,\,n\in \N$ and $\lim_{|x|\rightarrow \infty}\,V^{\mu_n}(x) = 0$ and, thus, $\lim_{|x|\rightarrow \infty}\,V^{\mu_Q}(x) = 0\,.$

For part (b), it suffices to take into account that now the equality in \eqref{equilibrium} holds for $|x|$ large enough. Therefore, multiplying this equality by $|x|^s$, taking limits as $|x|\rightarrow \infty$ and applying again the monotone convergence argument, it yields that necessarily $q=1$ and, hence, $F=0$, with the whole real axis being the support of the equilibrium measure.

\end{proof}

Now, Lemma \ref{lem:Fbehavior} (b) shows that for $q<1$, in case of existence of $\mu_Q$ it must be compactly supported. But in that case, for some $a$ large enough, $\mu_Q \equiv \mu_{a} = \mu_{Q,[-a,a]}\,,$ i.e. the weighted equilibrium measure for a compact interval $[-a,a]$. But, as seen in Section 3, for $q<1$ we have that $\eta_a \equiv \mu_a\,,\,$ for any $a>0$ (see Fig. 5) and, then, we would conclude that $\mu_Q \equiv \eta_a\,,\,$ for some $a>0$. Then, \eqref{signedfixed} shows that $$V^{\eta_a}(x) + Q(x) = (1- q\,m_a)\,W_s(a)\,,$$
with $W_s(a)>0$ being the $s$--energy of the interval $[-a,a]$ given in \eqref{energyint}. Finally, since $q<1$ and $m_a<1$, we would get a positive value for the equilibrium constant $F_Q$, and it contradicts Lemma \ref{lem:Fbehavior} (a).

\subsection{Proof of Theorem \ref{thm:single}}

We aim to determine the endpoints $\pm \widetilde{a}$ of the support $S_Q = [-\widetilde{a},\widetilde{a}]$. To this end, the Mhaskar--Saff $\mathcal{F}_s$-functional \eqref{MSfunct} will be used.
In our case, this functional takes the form:
\begin{equation}\label{MSa}
\mathcal{F}_s(K)\,=\,\mathcal{F}_s(a)\,=\,W_s(a)\,-\,q\,\int Q(x)\,d\omega_{[-a,a]}(x)\,=\,W_s(a)\,-\,q\,U_s^{\omega_{[-a,a]}}(b i)\,,
\end{equation}
with $d\omega_{[-a,a]}(x)$ given in \eqref{sequilmeasint} and $W_s(a)$ is given by \eqref{energyint}. Thus, after some calculations, we get
\begin{equation}\label{potential}
V_s^{\omega_{[-a,a]}}(b i)\, =\,\frac{B(\frac{1+s}{2},\frac{1}{2})}{2^2\,B(\frac{1+s}{2},\frac{1+s}{2})}\,(a^2+b^2)^{-s/2}\,
_2F_1\left(\frac{s}{2},\frac{1+s}{2};1+\frac{s}{2};\frac{a^2}{a^2+b^2} \right)\,,
\end{equation}
and therefore, from \eqref{energyint} and \eqref{potential}, \eqref{MSa} may be written in the form:
\begin{equation}\label{MSfinal}
\begin{split}
\mathcal{F}_s(a)\, & =\,\frac{\Gamma(1+s)}{2^s\,\Gamma\left(\frac{1+s}{2}\right)}\,a^{-s}\,\left(\Gamma\left(\frac{1-s}{2}\right)\,
-\,q\,\frac{\sqrt{\pi}}{\Gamma(1+\frac{s}{2})}\left(\frac{a}{\sqrt{a^2+b^2}}\right)^s\,_2F_1\left(\frac{s}{2},\frac{1+s}{2};1+\frac{s}{2};\frac{a^2}{a^2+b^2} \right)\right)\\
& =\,\frac{\Gamma(1+s)}{2^s\,\Gamma\left(\frac{1+s}{2}\right)}\,a^{-s}\,g(c,s)\,,
\end{split}
\end{equation}
with $\displaystyle c = \,\frac{a^2}{a^2+b^2}\,.$

Now, taking $b >0$ and $q>0$ fixed, it is easy to check (see \cite[Ch. 15]{Abramowitz}) that $g(c,s)$ in \eqref{MSfinal} is a decreasing function of $a$; and using \cite[(15.1.20)]{Abramowitz}, we also have that $$ _2F_1\left(\frac{s}{2},\frac{1+s}{2};1+\frac{s}{2};1 \right)\,=\,\frac{\Gamma\left(\frac{1+s}{2}\right)\,\Gamma\left(\frac{1-s}{2}\right)}{\sqrt{\pi}}$$ and, hence,
$$\lim_{a\rightarrow +\infty}\,g(c,s)\,=\,\Gamma\left(\frac{1-s}{2}\right)\,-\,q\,\frac{\sqrt{\pi}}{\Gamma(1+\frac{s}{2})}\,\lim_{c\rightarrow 1^-}\,_2F_1\left(\frac{s}{2},\frac{1+s}{2};1+\frac{s}{2};c\right)\,=
\,(1-q)\,\Gamma\left(\frac{1-s}{2}\right)\,.$$ Since, on the other hand, it is clear that
$$\lim_{a\rightarrow +\infty}\,\mathcal{F}_s(a) = 0\;\;\text{and}\;\;\lim_{a\rightarrow 0^+}\,\mathcal{F}_s(a) = +\infty\,,$$
it implies that for $q\leq 1$, $\mathcal{F}_s(a)$ is a strictly positive and decreasing function of $a$ and, thus, it has not a minimum. Otherwise, when $q>1$, $\mathcal{F}_s(a)<0$ for $a>a_0$, with $a_0 = a_0(s,q,b)$ and, since $\lim_{a\rightarrow +\infty}\,\mathcal{F}_s(a) = 0$, then $\mathcal{F}_s(a)$ must be an increasing function of $a$ for $a>\widetilde{a}$, with $\widetilde{a}=\widetilde{a}(s,q,b )\,>\,a_0\,.$ Therefore, the absolute minimum of the function $\mathcal{F}_s(a)$ is attained at this point $\widetilde{a}\in (0,+\infty)$ (see Figure \ref{fig:MS}).

\begin{figure}[h]
    \begin{center}
    \includegraphics[scale=0.4]{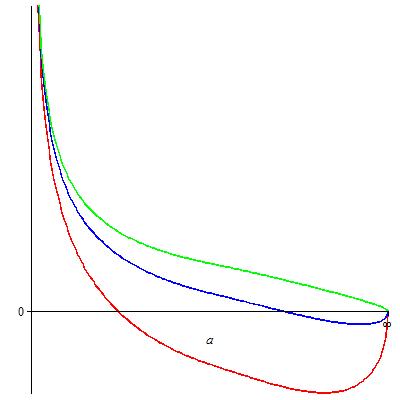}
    \end{center}
    \caption{Graph of the function $\mathcal{F}_s(a)$ for $s = 1/2, b = 1$, and $q = 5$ (red), $q=2$ (blue), $q=0.75$ (green). Observe that for $q=.75$, $\mathcal{F}_s$ is a positive and monotonically decreasing function in $(0,\infty)$ which never meets the real axis. }
    \label{fig:MS}
\end{figure}

%Finally, let us also point out that the uniqueness of the equilibrium measure guarantees, in turn, the uniqueness of the critical value of $a$ for which $m_a$ satisfies \eqref{criticalma}.

\subsection{Proof of Theorem \ref{thm:endbehavior}}

Indeed, from \eqref{densityeqmeas} and \eqref{limit} we have that
\begin{equation*}
\begin{split}
& \lim_{|x|\rightarrow \widetilde{a}^-}\,(\widetilde{a}^2 - x^2)^{-\,\frac{1+s}{2}}\,\mu_Q' (x)  = \\
& \lim_{|x|\rightarrow \widetilde{a}^-}\,(\widetilde{a}^2 - x^2)^{-\,\frac{1+s}{2}}\,
\left\{\frac{\mathcal{B}\left(\frac{1+s}{2},\frac{1-s}{2}\right)}{(x^2+b^2)^{1-\frac{s}{2}}} - \,(\widetilde{a}^2 - x^2)^{\frac{1+s}{2}}\,\int_0^{\infty}\,\frac{u^{-\,\frac{1+s}{2}}}{(u+\widetilde{a}^2+b^2)^{1-\frac{s}{2}}}\,
\frac{du}{u+\widetilde{a}^2-x^2} \right\} = \\
& \lim_{|x|\rightarrow \widetilde{a}^-}\,(\widetilde{a}^2 - x^2)^{-\,\frac{1+s}{2}}\,\int_0^{\infty}\,\frac{v^{-\frac{1+s}{2}}}{v+1}\,\left\{\frac{1}{(x^2+b^2)^{1-\frac{s}{2}}} - \,\frac{1}{((\widetilde{a}^2-x^2)v + \widetilde{a}^2+b^2)^{1-\frac{s}{2}}}\right\}\,dv.
\end{split}
\end{equation*}
Now, making the change of variable $\displaystyle w = \frac{\widetilde{a}^2-x^2}{x^2+b^2}\,(v+1)\,,$
we have that
\begin{equation*}
\begin{split}
& \lim_{|x|\rightarrow \widetilde{a}^-}\,(\widetilde{a}^2 - x^2)^{-\,\frac{1+s}{2}}\,\mu_Q' (x)  = \\
& \frac{1}{(x^2+b^2)^{1-\frac{s}{2}}}\,\lim_{|x|\rightarrow \widetilde{a}^-}\,(\widetilde{a}^2 - x^2)^{-\,\frac{1+s}{2}}\,\int_{\frac{\widetilde{a}^2-x^2}{x^2+b^2}}^{\infty}\,
\frac{\left(\frac{x^2+b^2}{\widetilde{a}^2-x^2}\,w\,-\,1\right)^{-\frac{1+s}{2}}\,\left(1-(1+w)^{\frac{s}{2}-1}\right)}{w}\,dw
\end{split}
\end{equation*}
Therefore, after some straightforward manipulations, we get
$$\lim_{|x|\rightarrow \widetilde{a}^-}\,(\widetilde{a}^2 - x^2)^{-\,\frac{1+s}{2}}\,\mu_Q' (x)  = \,\frac{1}{(\widetilde{a}^2+b^2)^{\frac{3}{2}}}\,\int_0^{\infty}\,\frac{1-(1+w)^{\frac{s}{2}-1}}{w^{\frac{3+s}{2}}}\,dw\,.$$
Since $\displaystyle 1-(1+w)^{\frac{s}{2}-1} = \left(1-\,\frac{s}{2}\right)\,w + O(w^2)\,$ for $w\rightarrow 0^+\,,$ we have that the last integral is convergent and does not vanish. This renders the proof of \eqref{endbehavior}.

\end{document}